\documentclass[11pt]{article}
\usepackage{amsmath, amsfonts, amssymb, amsthm, mathtools, amscd}
\usepackage{mathcmd, authblk}
\usepackage{mathrsfs} 
\usepackage{dsfont}
\usepackage[mathscr]{eucal}
\usepackage{eqnarray}
\usepackage{graphicx}
\usepackage{csquotes}
\usepackage{enumitem}

\usepackage[T1]{fontenc}
\usepackage{newtxtext, newtxmath} 

\usepackage[english]{babel}
\usepackage{microtype}
\usepackage{setspace}
\usepackage[a4paper, margin=1.25in]{geometry}
\usepackage{titling}

\pretitle{\begin{center}\LARGE\bfseries}
	\posttitle{\par\end{center}\vskip 0.5em}
\preauthor{\begin{center}\large}
	\postauthor{\end{center}}
\predate{\begin{center}\normalsize}
	\postdate{\end{center}}

\usepackage{titlesec}

\titleformat{\section}{\large\bfseries}{\thesection}{1em}{}
\titleformat{\subsection}{\normalsize\bfseries}{\thesubsection}{1em}{}

\newcommand{\Z}{{\mathbb Z}}
\newcommand{\Q}{{\mathbb Q}}

\newcommand{\abs}[1]{\left|#1\right|}

\renewcommand{\le}{\leqslant}
\renewcommand{\ge}{\geqslant}

\usepackage{hyperref} 
\hypersetup{
	colorlinks=true,      
	linkcolor=blue,        
	urlcolor=blue,         
	citecolor=blue        
}

\usepackage[capitalize]{cleveref}

\usepackage{aliascnt}

\newtheorem{thm}{Theorem}
\newtheorem*{thm*}{Theorem}

\newcommand{\newaliasedtheorem}[3]{%
	\newaliascnt{#1}{thm}%
	\newtheorem{#1}[#1]{#3}%
	\aliascntresetthe{#1}%
	\crefname{#1}{#2}{#2s}%
	\Crefname{#1}{#3}{#3s}%
}

\newaliasedtheorem{lem}{lemma}{Lemma}
\newaliasedtheorem{clm}{claim}{Claim}
\newaliasedtheorem{prop}{proposition}{Proposition}
\newaliasedtheorem{cor}{corollary}{Corollary}
\newaliasedtheorem{axm}{axiom}{Axiom}
\newaliasedtheorem{fact}{fact}{Fact}
\newaliasedtheorem{conj}{conjecture}{Conjecture}

\theoremstyle{remark}
\newaliasedtheorem{rem}{remark}{Remark}
\newaliasedtheorem{rmk}{remark}{Remark}
\newaliasedtheorem{note}{note}{Note}
\newaliasedtheorem{qs}{question}{Question}
\newaliasedtheorem{obs}{observation}{Observation}

\theoremstyle{definition}
\newaliasedtheorem{defn}{definition}{Definition}

\title{Lower Bounds for Moments of $L$-functions}
\author{Sanoli Gun, Gaurav Kumar, Deep Thakur}
\date{}

\begin{document}
\maketitle	
	
\begingroup
\renewcommand\thefootnote{}
\footnotetext{\textit{Address}[1] 
The Institute of Mathematical Sciences, C.I.T. Campus, Taramani, Chennai 600 113, India.}
\footnotetext{\textit{Address}[2]
Homi Bhabha National Institute, Training School Complex, Anushakti Nagar, Mumbai 400094, India.} 
\footnotetext{\textit{Emails.}  sanoli@imsc.res.in, gauravkr@imsc.res.in, deepthakur@imsc.res.in}
\endgroup
	
\begin{abstract}
In this article, we refine the method of Heap and Soundararajan \cite{HeapSoundararajan2022} by
introducing a convexity argument due to Heath-Brown \cite{HeathBrown1981} to obtain lower bounds 
for the $2k$-th moments of a broad class of $L$-functions for all real $k > 0$. 
In particular, this new technique circumvents the need to estimate twisted moments 
of $L$-functions. Hence, it allows us to deduce lower bounds of  the correct order of magnitude
for Dedekind zeta functions 
of arbitrary number fields, as well as for automorphic $L$-functions attached to irreducible 
unitary cuspidal representations of $\operatorname{GL}_d(\mathbb{A}_\mathbb{Q})$ 
unconditionally for $2 \le d \le 4$ and conditionally
on the Hypothesis H of Rudnick--Sarnak for all $d$.
\end{abstract}

\section{INTRODUCTION}
A central theme in analytic number theory is the study of moments of $L$-functions 
on the critical line, namely the quantity
$$
\int_T^{2T}\abs{L(\tfrac{1}{2}+it)}^{2k}dt,
$$
where $k > 0$ is real and $T>0$ is large. In the special case when the $L$-function is 
the Riemann zeta function, we have the following folklore conjecture for the $2k$-th moments.
\begin{conj}\label{conj:zeta}
For every real $k>0$, there is a constant $C_k>0$ such that
$$
\int_{T}^{2T}\abs{\zeta(\tfrac{1}{2}+it)}^{2k} \,dt \sim C_k T(\log T)^{k^2}
\quad\text{as}\quad
T\to\infty.
$$
\end{conj}
The classical works of Hardy--Littlewood \cite{HardyLittlewood1916} 
and Ingham \cite{Ingham1927b}  established this asymptotic formula in the cases 
$k = 1$ and $k= 2$, respectively, and 
these remain the only moments of the Riemann zeta function for which 
asymptotic formulas are known unconditionally. 
	
Following \cite[\S5.1]{iwaniec2004analytic}, we denote by $\mathcal S$ the
class of $L$-functions
\[
L(s,f)=\sum_{n=1}^\infty\frac{f(n)}{n^s}
=\prod_p\prod_{j=1}^{d}\Bigl(1-\frac{\alpha_j(p)}{p^s}\Bigr)^{-1}
\]
of degree $d$ satisfying the Ramanujan--Petersson conjecture, so that
$\abs{\alpha_j(p)}\le1$ for all primes $p$.
More generally, for $L$-functions in the class $\mathcal S$, one expects the following 
analogue of Conjecture~\ref{conj:zeta} to hold (see \cite{CFKRS05,Heap21,KS00}).
	
\begin{conj}\label{conj:heap}
Let $L(s, f)=\sum_{n = 1}^\infty f(n)n^{-s}$ be an element of $\mathcal S$, and 
suppose that, for some $N_f>0$, we have
$$
\sum_{p\le x}\frac{\abs{f(p)}^2}{p}=N_f\log\log x+O(1)
\quad\text{as}\quad
x\to\infty.
$$
Then, for every real $k>0$, there is a constant $C_f(k)>0$, such that
$$
\int_T^{2T}\abs{L(\tfrac{1}{2}+it,f)}^{2k}\,dt\sim C_f(k)\,T(\log T)^{N_fk^2}
\quad\text{as}\quad
T\to\infty.
$$
\end{conj}
	
Pioneering work of Ramachandra \cite{Ramachandra1978, Ramachandra1980a, Ramachandra1980b} 
and Heath-Brown \cite{HeathBrown1981} established lower bounds 
of the correct order of magnitude for the $2k$-th moments of the Riemann zeta function, 
unconditionally when $k > 0$ is rational and under the Riemann hypothesis when $k > 0$ is real. 
Heap and Soundararajan \cite{HeapSoundararajan2022} introduced a novel method to 
establish these lower bounds for the $2k$-th moment unconditionally for all 
real $k > 0$ for a wide family of $L$-functions (see also 
\cite{RadziwillSoundararajan2013, RudnickSoundararajan2005}). 
	
However, for $0 < k < 1$, the method of Heap and Soundararajan \cite{HeapSoundararajan2022} 
relies crucially on sharp upper bounds for twisted moments of the $L$-function
of the form
$$
\int_T^{2T}\abs{L(\tfrac{1}{2}+it)}^{2}\abs{\sum_{n\le T^\theta}\frac{a(n)}{n^{1/2+it}}}^2dt,
$$
where $\theta>0$ is small and $a(n)$ are complex numbers satisfying 
$a(n)\ll_\varepsilon n^\varepsilon$ for all $\varepsilon>0$. This requirement 
restricts the applicability of their method; for example, obtaining such sharp 
upper bounds for the twisted moments of $L$-functions attached to
automorphic representations of $\operatorname{GL}_d(\mathbb{A}_{\Q})$  is a difficult 
open problem when $d\ge 3$. 
	
In this article, we refine the method of Heap-Soundararajan \cite{HeapSoundararajan2022} 
by introducing a convexity argument due to Heath-Brown \cite{HeathBrown1981}, which circumvents 
the need for these twisted moment estimates. We illustrate this method for $L$-functions 
associated with multiplicative functions satisfying certain natural conditions, and 
we indicate in \Cref{sec:linearcombination} the modifications needed for linear combinations of
$L$-functions. Our main result is the following general theorem.
	
\begin{thm}\label{thm:multiplicativethm}
Let $f$ be a non-zero multiplicative function, and let
$$
L(s, f)=\sum_{n=1}^\infty\frac{f(n)}{n^s}
$$ 
be the $L$-function associated with $f$. Suppose that, for some 
$0<\theta<\tfrac{1}{2}$ and all integers $r\ge 1$, the function $f$ satisfies 
the estimates 
$$
f_{1/r}(n) \ll_{f,r} n^\theta\quad\text{and}\quad\sum_p\sum_{m=2}^\infty\frac{\abs{f_{1/r}(p^m)}^2}{p^m}<\infty,
$$
where $f_{1/r}(n)$ is defined formally by
$$
L(s, f)^{1/r}=\sum_{n=1}^\infty\frac{f_{1/r}(n)}{n^s}.
$$
Assume, furthermore that, for some $N_f>0$, we have
\begin{equation}
\label{eq:merten}
\sum_{p\le x}\frac{\abs{f(p)}^2}{p}=N_f\log\log x+O(1)
\quad\text{as}\quad
x\to\infty.
\end{equation}
Suppose also that $L(s, f)$ admits a meromorphic continuation to the half-plane 
$\sigma\ge \tfrac{1}{2}$, with at most one pole at $s=1$, and that, for some $A>0$, 
we have $L(s, f)=O(\abs{t}^A)$ as $\abs{t}\to\infty$, uniformly for $\sigma\ge \tfrac{1}{2}$. 
Then, for all real $k > 0$ and sufficiently large $T>0$, we have
$$
\int_T^{2T}\abs{L(\tfrac{1}{2}+it,f)}^{2k}\,dt\gg_{f,k} T(\log T)^{N_fk^2}.
$$
\end{thm}
When $k > 0$ is rational, \Cref{thm:multiplicativethm} can be deduced 
from the convexity argument of Heath-Brown (see \cite{AkbaryFodden2012}). 
\begin{rmk}
In view of \Cref{conj:heap}, \Cref{thm:multiplicativethm} gives lower
bounds of the conjectured order of magnitude for every
$L(s ,f)\in\mathcal{S}$ satisfying \eqref{eq:merten}. Indeed, the remaining
hypotheses of \Cref{thm:multiplicativethm} hold for every
$L(s, f)\in\mathcal{S}$. To see this, let $d$ denote the degree of
$L(s, f)$,  and let $\tau_d(n)$ denote the number of ways of writing $n$ as an 
ordered product of $d$ positive integers. Since for all integers $n,r\ge 1$,
the Ramanujan--Petersson bound gives 
$\abs{f_{1/r}(n)}\le\tau_d(n) \ll_{d,\varepsilon}n^{\varepsilon}$
for every $\varepsilon>0$, both
estimates on $f_{1/r}$ follow.
The analytic hypotheses are standard
(see \cite[Chapter~5]{iwaniec2004analytic}).
\end{rmk}
	
\begin{rmk}
\Cref{thm:multiplicativethm} does not require $L(s, f)$ to satisfy a
functional equation; in particular, its proof does not use an approximate
functional equation for $L(s, f)$ in the critical strip. Since the exponent
$N_fk^2$ in the lower bound is determined by \eqref{eq:merten} alone, 
the proof only sees, in some sense, the pole 
of the Rankin--Selberg convolution $L(s, f\times\overline{f})$ (if it exists) at $s=1$.
\end{rmk}
	
We now deduce the following corollaries from \Cref{thm:multiplicativethm}. 
	
\begin{cor}\label{cor:dedekindcor}
Let $K$ be a number field of degree $d$, and let $\zeta_{K}$ denote the Dedekind 
zeta function of $K$. Further, let $L$ be the normal closure of $K$ over $\mathbb{Q}$. 
Set $G = \operatorname{Gal}(L/\mathbb{Q})$ and $H=\operatorname{Gal}(L/K)$. 
Then, for all real $k > 0$ and all sufficiently large $T>0$, we have
$$
\int_T^{2T}\abs{\zeta_K(\tfrac{1}{2}+it)}^{2k}\,dt\gg_{K,k} T(\log T)^{M_K k^2},
$$
where $M_K = \# H \backslash G / H$. In particular, if $K/\mathbb{Q}$ is a 
Galois extension, then $M_K = d$.
\end{cor} 
\Cref{cor:dedekindcor} is known when $k>0$ is rational. When $K/\Q$ is Galois, 
it is also known for all real $k>1$ (see \cite{Krishnamoorthy2025, sono}).
	
\begin{cor}\label{cor:picor}
Let $\pi$ be an irreducible unitary cuspidal representation of 
$\operatorname{GL}_d(\mathbb{A}_\mathbb{Q})$, and let $L(s,\pi)$ be the 
$L$-function attached to $\pi$. Assume Hypothesis~H of 
Rudnick--Sarnak (see \cite[page 15]{RS}). Then, for all real $k > 0$ and all sufficiently
large $T>0$, we have
$$
\int_T^{2T}\abs{L(\tfrac{1}{2}+it,\pi)}^{2k}\,dt\gg_{\pi,k} T(\log T)^{k^2}.
$$
In particular, since Hypothesis~H is known for $d\le 4$ (see \cite[Theorem 2.1]{MTB} 
and also \cite{Kim,RS}), this bound holds unconditional for $d\le 4$.
\end{cor}
When $d=2$ and $\pi$ corresponds to a cuspidal Hecke eigenform 
for $\operatorname{SL}_2(\Z)$, the required bounds for the twisted second 
moments are available (see \cite[Lemma 4.1]{Chen2026}), and hence \Cref{cor:picor} 
is known in that case. When $d\ge 3$, however, the twisted second moments remain 
out of reach, and \Cref{cor:picor} is new for all $k>0$.
		
We postpone the proofs of the corollaries to \Cref{sec:corsec}. Throughout this article, 
the implicit constants in the notations $\ll, \gg, \asymp$ and $O(\cdot)$ may depend 
on the real number $k$ and on the multiplicative function $f$ (or, in \Cref{sec:linearcombination}, 
on the multiplicative functions $f_1, \ldots, f_d$ and the coefficients $b_1, \ldots, b_d$). Any dependence 
on further parameters is indicated explicitly by subscripts (e.g., $\ll_\varepsilon$).

\section{SETUP AND PLAN OF PROOF}\label{reqNot}
	
Throughout the article, $k$ denotes a positive real number, $T >0$ a sufficiently 
large real number, $f$ a non-zero multiplicative function as in 
Theorem \ref{thm:multiplicativethm}, 
and $C_f>0$ a sufficiently large constant 
depending only on $f$. Further, let  $\log_j$ denote the $j$-fold iterated logarithm 
and $\ell$ denote the largest integer such that $\log_{\ell}T\ge C_f$.
Define an increasing sequence $T_j$ as follows. 
Put $T_1=(k+1)^4C_f$, and for $2\le j\le \ell$, put 
$$
T_j=\exp\left(\frac{\log T}{(k+1)^2(\log_j T)^2}\right).
$$
For $2\le j\le \ell$, set 
$$
\mathcal{P}_j(s)=\sum_{T_{j-1}< p\le T_j}\frac{f(p)}{p^s}
\qquad \text{and} \qquad 
P_j=\sum_{T_{j-1}< p\le T_j}\frac{\abs{f(p)}^2}{p}.
$$
For each $2 \le j \le \ell$, put $K_j=500 (k+1)^2P_j$. Let 
$$
\mathcal{N}_j
~=~
\{n\in\mathbb{N} : p \mid n \implies p\in (T_{j-1},T_j]\text{ and } \Omega(n)\le K_j\}.
$$
and
$$
\mathcal{N}
~=~
\{n=n_2n_3\cdots n_{\ell}\in\mathbb{N} : n_j\in \mathcal{N}_j\}.
$$ 
Clearly for $n\in\mathcal{N}$, we have 
$$
n=n_2\cdots n_{\ell}\le T_2^{K_2}T_3^{K_3}\cdots T_{\ell}^{K_{\ell}}\le T^{1/9}.
$$
Let $g$ denote the multiplicative function given on prime powers by 
$\displaystyle{g(p^m)=\frac{f(p)^m}{m!}}$. 
For any real number $\alpha$ and $2\le j\le \ell$, define 
$$
\mathcal{N}_j(s,\alpha)
~=~
\sum_{0\le m\le K_j}\frac{1}{m!}(\alpha \mathcal{P}_j(s))^{m}
~=~
\sum_{\substack{p \mid n \implies T_{j-1} < p \le T_j\\
\Omega(n)\le K_j}}\frac{\alpha^{\Omega(n)}g(n)}{n^s},
$$
and put 
$$
\mathcal{N}(s,\alpha)=\prod_{j=2}^\ell \mathcal{N}_j(s,\alpha)
~=~
\sum_{n\in\mathcal{N}}\frac{\alpha^{\Omega(n)}g(n)}{n^s}.
$$
Let $r>1$ be the smallest integer such that $k>\tfrac2r$. To refine the method of 
\cite{HeapSoundararajan2022}, we consider the quantity
$$
\mathcal{I}(T)
=\int_{T}^{2T}\abs{L(\tfrac{1}{2}+it,f)}^{\frac2r}\abs{\mathcal{N}(\tfrac{1}{2}
+it,k-\tfrac{1}{r})}^2\,dt.
$$
By H\"older's inequality, we have  
$$
\mathcal{I}(T)\le\left( \int_T^{2T}\abs{L(\tfrac{1}{2}+it,f)}^{2k}\,dt\right)^{\frac{1}{rk}}
\left(\int_T^{2T}\abs{\mathcal{N}\left(\tfrac{1}{2}+it,k-\tfrac{1}{r}\right)}^{\frac{2rk}{rk-1}}\,dt\right)^{1-\frac{1}{rk}}.
$$
\Cref{thm:multiplicativethm} now follows immediately from the following two lemmas.
	
\begin{lem}\label{lem:multiNlem}
Let $T>0$ be sufficiently large. We have
$$
\int_T^{2T}\abs{\mathcal{N}\left(\tfrac{1}{2}+it,k-\tfrac{1}{r}\right)}^{\frac{2rk}{rk-1}}\,dt\ll T(\log T)^{N_fk^2}.
$$
\end{lem}
\begin{lem}\label{lem:multimainlem}
Let $T>0$ be sufficiently large. We have
$$
\mathcal{I}(T)\gg T(\log T)^{N_fk^2}.
$$
\end{lem}
Following the argument of Heap--Soundararajan \cite[Prop.~3]{HeapSoundararajan2022}, 
we prove \Cref{lem:multiNlem} in \Cref{sec:multiNlem} by 
replacing $\abs{\mathcal{N}(\tfrac{1}{2}+it,k-\tfrac{1}{r})}^{\frac{2rk}{rk-1}}$ 
with $\abs{\mathcal{N}(\tfrac{1}{2}+it,k)}^{2}$ on average, up to a manageable error term. 
A mean value estimate then gives
\begin{align*}
\int_T^{2T}\abs{\mathcal{N}(\tfrac{1}{2}+it,k)}^{2}\,dt
&\ll T\prod_{j=2}^{\ell}\sum_{\substack{p \mid n \implies T_{j-1} < p \le T_j\\ \Omega(n)\le K_j}}
\frac{k^{2\Omega(n)}\abs{g(n)}^2}{n}
\\
&\ll
T\prod_{T_1 < p \le T_\ell }\left( 1 + k^2 \frac{\abs{f(p)}^2}{p} \right)
\\
&\ll
T(\log T)^{N_fk^2}.
\end{align*}
	
In \Cref{sec:multimainlem}, we prove \Cref{lem:multimainlem}   
by applying a convexity argument due to
Heath-Brown \cite{HeathBrown1981} to the integrals
$$
\mathcal{J}(\sigma) = \int_{-\infty}^{\infty}\abs{L(\sigma+it,f)}^{\frac2r}
\abs{\mathcal{N}(\sigma+it,k-\tfrac{1}{r})}^2 w(t)\,dt,
$$
$$
\mathcal{L}(\sigma) = \int_{-\infty}^{\infty}\abs{S(\sigma+it)}^{2} w(t)\,dt,
$$
where the weight $w(t)>0$ is bounded and decays rapidly outside the interval $[T,2T]$, and 
$S(s)$ is a short Dirichlet polynomial chosen so that 
$$
E(s) = L(s,f)\mathcal{N}(s,k-\tfrac{1}{r})^r-S(s)^r
$$
satisfies the pointwise bound $E(s) \ll T^{-\varepsilon}$ for some $\varepsilon>0$, 
uniformly on the line $\Re s=2$. It follows from the rapid decay of $w(t)$ that
$$
\mathcal{J}(\tfrac12)\ll \mathcal{I}(T).
$$
For the quantity $\mathcal{L}(\sigma)$, we prove that
$$
\mathcal{L}(\tfrac{1}{2}) \asymp T(\log T)^{N_fk^2}
~~\text{ and }~~
\mathcal{L}(\sigma) \asymp T(\sigma-\tfrac{1}{2})^{-N_fk^2}
$$
the latter uniformly for $\tfrac{1}{2} + \tfrac{1}{\log T} \le \sigma \le 1$. Finally, 
using a convexity theorem for integrals (applied to $L\mathcal{N}^r$ 
and to $E$), we deduce a bound of the shape
$$
\mathcal{L}(\sigma) \ll T^{(\sigma-1/2)} \mathcal{J}(\tfrac12) 
+ T^{-\delta(\sigma-1/2)}\mathcal{L}(\tfrac{1}{2})
$$
for some $\delta>0$.
Choosing $\sigma = \tfrac{1}{2} + \tfrac{c}{\log T}$ for $c\ge 1$, we get 
$$
T(\log T)^{N_fk^2} c^{-N_fk^2} \ll  e^{c}\mathcal{J}(\tfrac12)  + T(\log T)^{N_fk^2} e^{ -\delta c}.
$$
Hence, taking $c>0$ sufficiently large, we conclude that 
$$
\mathcal{I}(T) \gg \mathcal{J}(\tfrac12) \gg  T(\log T)^{N_fk^2}.
$$
	
\section{PROOF OF LEMMA \ref{lem:multiNlem}}\label{sec:multiNlem}
We need the following two intermediary lemmas.
	
\begin{lem}\label{lem:Npowerlem}
For $2\le j \le \ell$
$$
\abs{\mathcal{N}_j(\tfrac{1}{2}+it,k-\tfrac{1}{r})}^{\frac{2rk}{rk-1}}
\le
\abs{\mathcal{N}_j(\tfrac{1}{2}+it,k)}^2(1+5e^{-K_j})+\mathcal{Q}_j(t)
$$
where 
$$
\mathcal{Q}_j(t)=\left(\frac{12(k+1)\abs{\mathcal{P}_j(\tfrac{1}{2}+it)}}{K_j}\right)^{4K_j}.
$$
\end{lem}
\begin{proof}[Proof of \Cref{lem:Npowerlem}]		
For $\abs{z}\le \tfrac{K}{10}$, we have		
\begin{equation}\label{approximation}				
\abs{\sum_{m=0}^K\frac{z^m}{m!}-e^z}\le \frac{\abs{z}^K}{K!}\le \left(\frac{e}{10}\right)^K.		
\end{equation}		
In the first case, where $\abs{\mathcal{P}_j(\tfrac{1}{2}+it)}\le \tfrac{K_j}{10k+10}$, 
we have by \eqref{approximation} that		
\begin{align*}			
\abs{\mathcal{N}_j(\tfrac{1}{2}+it,k-\tfrac{1}{r})}^{\frac{2rk}{rk-1}}			
&
=\abs{\sum_{0\le m\le K_j}\frac{1}{m!}((k-\tfrac{1}{r})\mathcal{P}_j(\tfrac{1}{2}+it))^m}^{\frac{2rk}{rk-1}}\\			
&
\le	\abs{\exp\left(2k\mathcal{P}_j(\tfrac{1}{2}+it)\right)}(1+e^{-K_j})^{\frac{2rk}{rk-1}}\\			
&
\le	\abs{\mathcal{N}_j(\tfrac{1}{2}+it,k)}^2(1+5e^{-K_j}).		
\end{align*}		
The lemma follows in that case.	 In the second case, where 
$\abs{\mathcal{P}_j(\tfrac{1}{2}+it)} > \tfrac{K_j}{10k+10}$, we have 	
\begin{align*}			
\abs{\mathcal{N}_j(\tfrac{1}{2}+it,k-\tfrac{1}{r})}			
&\le			
\sum_{0\le m\le K_j}\frac{1}{m!}((k-\tfrac{1}{r}) \abs{\mathcal{P}_j(\tfrac{1}{2}+it)})^m\\			
&\le			
\abs{(k+1) \mathcal{P}_j(\tfrac{1}{2}+it)}^{K_j}\sum_{0\le m\le K_j}\frac{1}{m!}(\tfrac{10}{K_j})^{K_j-m} \\			
&\le	
\left(\frac{12(k+1)\abs{\mathcal{P}_j(\tfrac{1}{2}+it)}}{K_j}\right)^{K_j}.	
\end{align*}		
Hence	
\begin{align*}				
\abs{\mathcal{N}_j(\tfrac{1}{2}+it,k-\tfrac{1}{r})}^{\frac{2rk}{rk-1}}		
&~\le~		
\left(\frac{12(k+1)\abs{\mathcal{P}_j(\tfrac{1}{2}+it)}}{K_j}\right)^{2rkK_j/(rk-1)}\\	
&~\le~		
\left(\frac{12(k+1)\abs{\mathcal{P}_j(\tfrac{1}{2}+it)}}{K_j}\right)^{4K_j}.	
\end{align*}
\end{proof}
	
\begin{lem}\label{lem:Qlem}
With the above notation and for sufficiently large $T>0$, we have
$$
\int_{T}^{2T}\mathcal{Q}_j(t)\,dt\ll Te^{-K_j}.
$$
\end{lem}
	
\begin{proof}[Proof of \Cref{lem:Qlem}]		
We have 		
$$		
\mathcal{P}_j(s)^{2K_j}=\sum_{\substack{p \mid n \implies T_{j-1} < p \le T_j\\\Omega(n)= 2K_j}}\frac{(2K_j)!g(n)}{n^{s}}.
$$				
It is a short Dirichlet polynomial, since $T_j^{2K_j}\le T^{2/9}$. By the mean value theorem for 
Dirichlet polynomials, we have 		
\begin{align*}			
\int_{T}^{2T}\mathcal{Q}_j(t)\,dt			
&=			
\left(\frac{12(k+1)}{K_j}\right)^{4K_j}((2K_j)!)^2\int_T^{2T}			
\abs{\sum_{\substack{p \mid n \implies T_{j-1} < p \le T_j\\ \Omega(n)= 2K_j}}\frac{g(n)}{n^{1/2+it}}}^2\,dt \\			
&=\left(\frac{12(k+1)}{K_j}\right)^{4K_j}((2K_j)!)^2 (T+O(\sqrt{T}))
\sum_{\substack{p \mid n \implies T_{j-1} < p \le T_j\\ \Omega(n)= 2K_j}}\frac{\abs{g(n)}^2}{n}.
\end{align*}		
Since $\abs{g(p^a)}^2=\tfrac{\abs{f(p)}^{2a}}{(a!)^2}\le\tfrac{\abs{f(p)}^{2a}}{a!}$, 
by expanding $P_j^{2K_j}$, we get		
$$		
\sum_{\substack{p \mid n \implies T_{j-1} < p \le T_j\\				
\Omega(n)= 2K_j}}\frac{\abs{g(n)}^2}{n}		
~\le~		
\frac{1}{(2K_j)!}\left(\sum_{T_{j-1}<p\le T_j}\frac{\abs{f(p)}^2}{p}\right)^{2K_j}		
~=~	
\frac{P_j^{2K_j}}{(2K_j)!}.		
$$	
Combining the above with the bound $(2K_j)!\le(2K_j)^{2K_j}$ and the
relation $P_j= \tfrac{K_j}{500(k+1)^2}$, we obtain
\begin{align*}
\int_{T}^{2T}\mathcal{Q}_j(t)\,dt
&~\le ~
\left(\frac{12(k+1)}{K_j}\right)^{4K_j}(2K_j)! (T+O(\sqrt{T}))P_j^{2K_j}\\
&~\le~
\left(\frac{4\cdot 12^4}{500^2}\right)^{K_j}(T+O(\sqrt{T}))
~\ll~ Te^{-K_j}.
\end{align*}
\end{proof}
We now deduce \Cref{lem:multiNlem} from \Cref{lem:Npowerlem} and \Cref{lem:Qlem} as follows.
\begin{proof}[Proof of \Cref{lem:multiNlem}]
By \Cref{lem:Npowerlem}, we have
$$
\abs{\mathcal{N}(\tfrac{1}{2}+it,k-\tfrac{1}{r})}^{\frac{2rk}{rk-1}}
\le
\prod_{j=2}^{\ell}\left(\abs{\mathcal{N}_j(\tfrac{1}{2}+it,k)}^2(1+5e^{-K_j})+\mathcal{Q}_j(t)\right).
$$
For $2\le j\le \ell$ the Dirichlet polynomials $\mathcal{N}_j(s,k)$ 
and $\mathcal{P}_j(s)^{2K_j}$ are supported on integers whose 
prime factors lie in the pairwise disjoint intervals $(T_{j-1},T_j]$. 
Hence, by the mean value theorem for Dirichlet polynomials (cf. \cite[p.~10-11]{HeapSoundararajan2022}), 
we obtain
\begin{align*}
&\int_T^{2T}\abs{\mathcal{N}\left(\tfrac{1}{2}+it,k-\tfrac{1}{r}\right)}^{\frac{2rk}{rk-1}}\,dt \\
&\ll~
T\prod_{j=2}^{\ell}\left(\frac{1}{T}\int_{T}^{2T}
\left(\abs{\mathcal{N}_j\left(\tfrac{1}{2}+it,k\right)}^2(1+5e^{-K_j})
+\mathcal{Q}_j(t)\right)\,dt\right).
\end{align*}
By the mean value theorem for Dirichlet polynomials, we have
\begin{align*}
\int_{T}^{2T}\abs{\mathcal{N}_j\left(\tfrac{1}{2}+it,k\right)}^2\,dt
&=
(T+O(\sqrt{T}))\sum_{\substack{p \mid n \implies T_{j-1} < p \le T_j\\ \Omega(n)\le K_j}}
\frac{k^{2\Omega(n)}\abs{g(n)}^2}{n}\\
&\le
(T+O(\sqrt{T}))\prod_{T_{j-1} < p \le T_j}\left(1+\frac{k^2\abs{f(p)}^2}{p}
+O\left(\frac{\abs{f(p)}^2}{p^{2-2\theta}}\right)\right).
\end{align*}
Here we have used the bound $\tfrac{\abs{f(p)}^2}{p}\ll p^{2\theta-1}$, so that
\begin{align*}
\sum_{a\ge2}\frac{k^{2a}\abs{f(p)}^{2a}}{(a!)^2p^a}
&=\frac{\abs{f(p)}^2}{p}\sum_{a\ge2}\frac{k^{2a}}{(a!)^2}\left(\frac{\abs{f(p)}^2}{p}\right)^{a-1}\\
&\ll\frac{\abs{f(p)}^2}{p^{2-2\theta}}.
\end{align*}
Using this and \Cref{lem:Qlem}, we get
\begin{align*}
\frac{1}{T}\int_{T}^{2T}
&\left(\abs{\mathcal{N}_j\left(\tfrac{1}{2}+it,k\right)}^2(1+5e^{-K_j})
+\mathcal{Q}_j(t)\right)\,dt	\\
&\qquad\qquad\qquad\qquad\qquad\le
(1+O(e^{-K_j}))\prod_{T_{j-1} < p \le T_j}
\left(1+\frac{k^2\abs{f(p)}^2}{p}+O\left(\frac{\abs{f(p)}^2}{p^{2-2\theta}}\right)\right).
\end{align*}
Hence
\begin{align*}
\int_T^{2T}\abs{\mathcal{N}\left(\tfrac{1}{2}+it,k-\tfrac{1}{r}\right)}^{\frac{2rk}{rk-1}}\,dt
&\ll~
T\prod_{p \le T}\left(1+\frac{k^2\abs{f(p)}^2}{p}+O\left(\frac{\abs{f(p)}^2}{p^{2-2\theta}}\right)\right)\\
&\ll~
T(\log T)^{N_fk^2}.
\end{align*}
\end{proof}
	
\section[Proof of Lemma 3]{PROOF OF LEMMA \ref{lem:multimainlem}}
\label{sec:multimainlem}
Let $N=T^{1/2}$ and define the coefficients $b(l)$ formally by
$$
L(s, f)^{1/r}\mathcal{N}(s,k-\tfrac{1}{r})=\sum_{l=1}^{\infty}\frac{b(l)}{l^s},
$$
so that
$$
b(l)=\sum_{\substack{mn=l\\ n\in\mathcal{N}}}f_{1/r}(m)(k-\tfrac{1}{r})^{\Omega(n)}g(n).
$$
Let
$$
S(s)=\sum_{l\le N}\frac{b(l)}{l^s}.
$$
Let 
$$
E(s)= L(s, f)\mathcal{N}(s,k-\tfrac{1}{r})^r-S(s)^r.
$$
Let
$$
w(t)=\int_{5T/4}^{7T/4}e^{-2(t-\tau)^2/r}\,d\tau,
$$
and
$$
\mathcal{J}(\sigma)=\int_{-\infty}^{\infty}\abs{L(\sigma+it,f)}^{2/r}\abs{\mathcal{N}(\sigma+it,k-\tfrac{1}{r})}^2w(t)\,dt,
$$
$$
\mathcal{K}(\sigma)=\int_{-\infty}^{\infty}\abs{E(\sigma+it)}^{2/r}w(t)\,dt,
$$
$$
\mathcal{L}(\sigma)=\int_{-\infty}^{\infty}\abs{S(\sigma+it)}^{2}w(t)\,dt.
$$
Since 
\begin{align*}
\abs{S(s)}^{2}&=\abs{L(s,f)\mathcal{N}(s,k-\tfrac{1}{r})^r-E(s)}^{2/r}\\
&\ll\abs{L(s,f)}^{2/r}\abs{\mathcal{N}(s,k-\tfrac{1}{r})}^{2}+\abs{E(s)}^{2/r},
\end{align*}
it follows that 
\begin{equation}\label{eq:I}
\mathcal{L}(\sigma)\ll \mathcal{J}(\sigma)+\mathcal{K}(\sigma),
\end{equation}
and similarly, 
\begin{equation}\label{eq:II}
\mathcal{K}(\tfrac{1}{2})\ll \mathcal{J}(\tfrac{1}{2})+\mathcal{L}(\tfrac{1}{2}).
\end{equation}
	
The following convexity result for integrals is due to Gabriel \cite{Gabriel1927}.
	
\begin{lem}\label{lem:gabriel1}
Let $F(z)$ be holomorphic in the infinite strip $\alpha < \Re(z) < \beta$, 
and continuous for $\alpha \le \Re(z) \le \beta$. Suppose $F(z) \to 0$ 
as $\abs{\Im(z)} \to \infty$ uniformly for $\alpha \le \Re(z) \le \beta$. 
Then for $\alpha \le \gamma \le \beta$ and any $q > 0$, we have
$$
\int_{-\infty}^{\infty} \abs{F(\gamma + it)}^q \, dt 
\le \left\{ \int_{-\infty}^{\infty} \abs{F(\alpha + it)}^q \, dt \right\}^{(\beta - \gamma)/(\beta - \alpha)} 
\left\{ \int_{-\infty}^{\infty} \abs{F(\beta + it)}^q \, dt \right\}^{(\gamma - \alpha)/(\beta - \alpha)}.
$$
\end{lem}
Applying \Cref{lem:gabriel1}, we show that $\mathcal{J}(\sigma)$ and $\mathcal{K}(\sigma)$ 
satisfy the following convexity estimates.

\begin{lem}\label{lem:JKlem}
Let $\tfrac12\le \sigma\le \tfrac34$ and $T>0$ be sufficiently large. Then 
$$
\mathcal{J}(\sigma)\ll T^{\sigma-1/2}\mathcal{J}(\tfrac{1}{2})^{3/2-\sigma}+e^{-T^2/(4r)},
$$
and
$$
\mathcal{K}(\sigma)\ll T^{\frac{2\sigma-1}{3}}N^{-\frac{2\sigma-1}{3r}}
\mathcal{K}(\tfrac{1}{2})^{\frac{4-2\sigma}{3}}+e^{-T^2/(4r)}.
$$
\end{lem}
\begin{proof}[Proof of \Cref{lem:JKlem}] 
To prove the estimate for $\mathcal{J}(\sigma)$ in \Cref{lem:JKlem}, let
$\tfrac12\le\sigma\le \tfrac34$ and $\tfrac{5T}{4}\le \tau \le \tfrac{7T}{4}$.
We apply \Cref{lem:gabriel1} with $\gamma=\sigma$, $\alpha=\tfrac12$, $\beta=\tfrac32$,
$q=\tfrac2r$ and
$$
F(z)= (z-1)^{m}L(z,f)\,\mathcal{N}(z,k-\tfrac{1}{r})^r\exp((z-i\tau)^2),
$$
where $m\ge 0$ is the order of the pole of $L(z,f)$ at $z=1$.
We have
\begin{align*}
\int_{-\infty}^\infty\abs{F(\tfrac{1}{2}+it)}^{\frac2r}\,dt&\ll\int_{\tau/2}^{3\tau/2}
\abs{F(\tfrac{1}{2}+it)}^{\frac2r}\,dt ~+~ e^{\frac{-2\tau^2}{5r}}\\
&\ll\tau^{\frac{2m}{r}}\int_{\tau/2}^{3\tau/2}\abs{L(\tfrac{1}{2}+ it, f)}^{\frac2r}
\abs{\mathcal{N}(\tfrac{1}{2}+it,k-\tfrac{1}{r})}^2e^{\frac{-2(t-\tau)^2}{r}}\,dt \\
&\qquad\qquad\qquad\qquad\qquad\qquad\qquad\qquad\qquad\qquad\qquad+e^{\frac{-2\tau^2}{5r}}.
\end{align*}
Similarly, we have
\begin{align}
\int_{-\infty}^\infty\abs{F(\tfrac{3}{2}+it)}^{\frac2r}\,dt
&\ll\tau^{\frac{2m}{r}}\int_{\tau/2}^{3\tau/2}\abs{L(\tfrac{3}{2}+it,f)}^{\frac2r}
\abs{\mathcal{N}(\tfrac{3}{2}+it, k-\tfrac{1}{r})}^2e^{\frac{-2(t-\tau)^2}{r}}\,dt\notag\\
&\qquad\qquad\qquad\qquad\qquad\qquad\qquad\qquad\qquad\qquad+~~e^{\frac{-2\tau^2}{5r}}\notag\\
&\ll\tau^{\frac{2m}{r}}\int_{\tau/2}^{3\tau/2}e^{\frac{-2(t-\tau)^2}{r}}\,dt~+~e^{\frac{-2\tau^2}{5r}}\notag\\
&\ll \tau^{\frac{2m}{r}}.\notag
\end{align}
We conclude from \Cref{lem:gabriel1} that
\begin{align}
&\int_{-\infty}^{\infty}\abs{F(\sigma+it)}^{\frac2r}\,dt \notag \\
&\ll \tau^{\frac{2m}{r}}\left(\int_{-\infty}^{\infty}\abs{L(\tfrac{1}{2}+it,f)}^{\frac2r}
\abs{\mathcal{N}(\tfrac{1}{2}+it, k-\tfrac{1}{r})}^2e^{\frac{-2(t-\tau)^2}{r}}\,dt\right)^{\frac32-\sigma}\notag\\
&\qquad\qquad\qquad\qquad\qquad\qquad\qquad\qquad\qquad\qquad+ ~~e^{\frac{-\tau^2}{4r}}\label{eq:Jlemeq1}.
\end{align}
However
\begin{align}
\int_{-\infty}^\infty\abs{L(\sigma+it,f)}^{\frac2r}&
\abs{\mathcal{N}(\sigma+it,k-\tfrac{1}{r})}^2e^{\frac{-2(t-\tau)^2}{r}}\,dt\notag\\
&\ll\int_{\tau/2}^{3\tau/2}\abs{L(\sigma+it,f)}^{\frac2r}
\abs{\mathcal{N}(\sigma+it,k-\tfrac{1}{r})}^2e^{\frac{-2(t-\tau)^2}{r}}\,dt~+~e^{\frac{-2\tau^2}{5r}}\notag\\
&\ll\tau^{\frac{-2m}{r}}\int_{\tau/2}^{3\tau/2}\abs{F(\sigma+it)}^{\frac2r}\,dt
~+~e^{\frac{-2\tau^2}{5r}}\notag\\
&\ll\tau^{\frac{-2m}{r}}\int_{-\infty}^{\infty}\abs{F(\sigma+it)}^{\frac2r}\,dt
~+~ e^{\frac{-2\tau^2}{5r}}.\label{eq:Jlemeq2}
\end{align}
The estimate for $\mathcal{J}(\sigma)$ in \Cref{lem:JKlem} now follows 
by combining \eqref{eq:Jlemeq1} with \eqref{eq:Jlemeq2} and 
integrating over $\tfrac{5T}{4}\le \tau \le \tfrac{7T}{4}$.

To prove the estimate for $\mathcal{K}(\sigma)$ in \Cref{lem:JKlem}, we apply \Cref{lem:gabriel1} 
to the function 
$$
F(z)=(z-1)^mE(z)\exp((z-i\tau)^2),
$$
with $\gamma=\sigma$, $\alpha=\tfrac12$, $\beta=2$, $q=\tfrac2r$, 
where $\tfrac12\le\sigma\le \tfrac34$ and $\tfrac{5T}{4}\le \tau \le \tfrac{7T}{4}$. 
For $\sigma\ge\tfrac32$ we have 
$$
\biggl(\sum_l \frac{b(l)}{l^s}\biggr)^r=L(s,f)\mathcal{N}(s,k-\tfrac{1}{r})^r.
$$ 
Since $S(s)$ consists of the terms of $\displaystyle{\sum_l b(l)l^{-s}}$ with $l\le N$, we get
$$
E(s)=\Big(\sum_{l=1}^{\infty}\frac{b(l)}{l^s}\Big)^r-\Big(\sum_{l\le N}\frac{b(l)}{l^s}\Big)^r=\sum_{l>N}\frac{e(l)}{l^s}\qquad(\sigma\ge\tfrac32).
$$
Since there exists $\theta'$ with $\theta<\theta'<\tfrac12$
such that $\abs{b(l)}\ll l^{\theta'}$, we have
$$
\abs{e(l)}\le\sum_{l_1\cdots l_r=l}\abs{b(l_1)}\cdots\abs{b(l_r)}
\ll l^{\theta'}\tau_{r}(l).
$$
It follows that
$$
\abs{E(2+it)}\le N^{-1/2}\sum_{l>N}\frac{\abs{e(l)}}{l^{3/2}}\ll N^{-1/2}
\sum_{l=1}^{\infty}\frac{l^{\theta'}\tau_{r}(l)}{l^{3/2}}\ll N^{-1/2},
$$
and therefore
$$
\int_{-\infty}^{\infty}\abs{F(2+it)}^{\frac2r}\,dt
\ll N^{\frac{-1}{r}}\int_{-\infty}^{\infty}\abs{1+it}^{\frac{2m}{r}}e^{\frac{2(4-(t-\tau)^2)}{r}}\,dt\ll 
N^{\frac{-1}{r}}\tau^{\frac{2m}{r}}.
$$
As in the proof of the estimate for $\mathcal{J}(\sigma)$, we have
$$
\int_{-\infty}^\infty\abs{F(\tfrac{1}{2}+it)}^{\frac2r}\,dt
\ll\tau^{\frac{2m}{r}}\int_{-\infty}^{\infty}\abs{E(\tfrac{1}{2}+it)}^{\frac2r}e^{\frac{-2(t-\tau)^2}{r}}\,dt
~+~ e^{\frac{-2\tau^2}{5r}}
$$
and
$$
\int_{-\infty}^\infty\abs{E(\sigma+it)}^{\frac2r}e^{\frac{-2(t-\tau)^2}{r}}\,dt
\ll\tau^{\frac{-2m}{r}}\int_{-\infty}^{\infty}\abs{F(\sigma+it)}^{\frac2r}\,dt
~+~ e^{\frac{-2\tau^2}{5r}}.
$$
Hence \Cref{lem:gabriel1} yields
\begin{align*}
&\int_{-\infty}^\infty\abs{E(\sigma+it)}^{\frac2r}e^{\frac{-2(t-\tau)^2}{r}}\,dt\\
&\qquad\ll N^{-\frac{2\sigma-1}{3r}}\left(\int_{-\infty}^\infty
\abs{E(\tfrac{1}{2}+it)}^{\frac2r}
e^{\frac{-2(t-\tau)^2}{r}}\,dt\right)^{\frac{4-2\sigma}{3}}~+~~ e^{\frac{-\tau^2}{4r}}.
\end{align*}
The estimate for $\mathcal{K}(\sigma)$ in \Cref{lem:JKlem} now 
follows by integrating over $\tfrac{5T}{4}\le \tau \le \tfrac{7T}{4}$ 
and applying H\"older's inequality.
\end{proof}
We now turn our attention to estimating $\mathcal{L}(\sigma)$ close to the critical line $\sigma =\tfrac12$.
\begin{lem}\label{lem:Llem}
Let $T>0$ be sufficiently large. We have
$$
\mathcal{L}(\sigma)\asymp T(\sigma-\tfrac{1}{2})^{-N_fk^2}
$$
uniformly for 
$\frac{1}{2}+\frac{1}{\log T}\le \sigma\le 1$.
Moreover,
$$
\mathcal{L}(\tfrac{1}{2})\asymp T(\log T)^{N_fk^2}.
$$
\end{lem}
\begin{proof}[Proof of \Cref{lem:Llem}]
We note that 
$$
w(t)\ll \exp(-(T^2+t^2)/(18r))
$$
for $t\le 0$ as well as for $t\ge 3T$. Hence
$$
\mathcal{L}(\sigma)=\int_0^{3T}\abs{S(\sigma+it)}^{2}w(t)\,dt+O(1).
$$
Moreover $w(t)\ll 1 $ for all $t$, and $w(t)\gg 1 $ for $\tfrac{4T}{3}\le t\le \tfrac{5T}{3}$. 
Hence we have 
$$
\mathcal{L}(\sigma)\gg \int_{4T/3}^{5T/3}\abs{S(\sigma+it)}^{2}\,dt.
$$
Similarly
$$
\mathcal{L}(\sigma)\ll  \int_{0}^{3T}\abs{ S(\sigma+it)}^{2}\,dt.
$$
Since the length of the polynomial $S(s)$ is $N=T^{1/2}$, 
by the mean value theorem \cite[Cor. 3]{meanvalue} each of these two integrals is
$$
\asymp T\sum_{l\le N} \frac{\abs{b(l)}^2}{l^{2\sigma}}.
$$
By the triangle inequality and multiplicativity of $f_{1/r}(n)$ and $g(n)$, we have 
\begin{align*}
\sum_{l\le N} \frac{\abs{b(l)}^2}{l^{2\sigma}}&\le \sum_{l}
 \left(\sum_{\substack{mn=l \\ p \mid m \implies p \le T\\ p \mid n \implies p \le T}} 
 \frac{\abs{f_{1/r}(m)}(k-\tfrac{1}{r})^{\Omega(n)}\abs{g(n)}}{(mn)^\sigma}\right)^2\\
& \le \prod_{p\le T} \left( \sum_{\substack{v_1, v_2, u_1, u_2 \ge 0 \\ v_1 + u_1 = v_2 + u_2}}
 \frac{(k-\tfrac{1}{r})^{u_1+u_2}\abs{f_{1/r}(p^{v_1})} \abs{f_{1/r}(p^{v_2})} \abs{g(p^{u_1})} 
 \abs{g(p^{u_2})}}{p^{2\sigma(v_1 + u_1)}} \right)\\
&=\prod_{p\le T} \left( 1+k^2\frac{\abs{f(p)}^2}{p^{2\sigma}}+E_p(\sigma) \right),
\end{align*}
where 
$$
E_p(\sigma) = \sum_{j=2}^\infty \frac{1}{p^{2\sigma j}} 
\left( \sum_{v+u=j} (k-\tfrac{1}{r})^u \abs{f_{1/r}(p^v)} \frac{\abs{f(p)}^u}{u!} \right)^2.
$$
We will show that 
$$
\sum_{p}E_p(\sigma)\le\sum_p E_p(\tfrac{1}{2})<\infty.
$$
By the Cauchy--Schwarz inequality, we have
\begin{align*}
\sum_p E_p(\tfrac{1}{2}) &\le \sum_p\sum_{j=2}^\infty \frac{j+1}{p^j} 
\sum_{v+u=j} (k-\tfrac{1}{r})^{2u} \abs{f_{1/r}(p^v)}^2 \frac{\abs{f(p)}^{2u}}{(u!)^2}\\
&=\sum_p\sum_{\substack{v, u \ge 0 \\ v+u \ge 2}} (v+u+1) 
\frac{\abs{f_{1/r}(p^v)}^2}{p^v} \frac{(k-\tfrac{1}{r})^{2u} \abs{f(p)}^{2u}}{p^u (u!)^2}.
\end{align*}
The contribution of the terms with $u = 0$ and $v \ge 2$ is
$$
\sum_p\sum_{v=2}^\infty (v+1) \frac{\abs{f_{1/r}(p^v)}^2}{p^v}.
$$
We have $\tfrac{\abs{f_{1/r}(p^v)}^2}{p^v} \ll p^{-(1 - 2\theta)v}$. 
We choose $V>0$ such that $(1 - 2\theta)V > 4$. Then we have
\begin{align*}
\sum_p\sum_{v=2}^{\infty} (v+1) \frac{\abs{f_{1/r}(p^v)}^2}{p^v} 
&\le V\sum_p \sum_{2\le v\le V} \frac{\abs{f_{1/r}(p^v)}^2}{p^v}
+\sum_p\sum_{v>V} (v+1)\frac{\abs{f_{1/r}(p^v)}^2}{p^v}\\
&\le V\sum_p \sum_{2\le v\le V} \frac{\abs{f_{1/r}(p^v)}^2}{p^v}
+\sum_p\sum_{v>V} \frac{v+1}{p^{(1 - 2\theta)v}} \\
&\le V\sum_p \sum_{v=2}^\infty \frac{\abs{f_{1/r}(p^v)}^2}{p^v}+\sum_p \frac{1}{p^2}<\infty.
\end{align*}
The contribution of the terms with $v = 0$ and $u \ge 2$ is
$$
\sum_p\sum_{u=2}^\infty (u+1) \frac{(k-\tfrac{1}{r})^{2u}}{(u!)^2} \frac{\abs{f(p)}^{2u}}{p^u}.
$$
For $u \ge 2$, we have
$$
\frac{\abs{f(p)}^{2u}}{p^u} \ll \frac{\abs{f(p)}^2}{p^{2-2\theta}}.
$$
Therefore,
\begin{align*}
\sum_p\sum_{u=2}^\infty (u+1) \frac{(k-\tfrac{1}{r})^{2u}}{(u!)^2} \frac{\abs{f(p)}^{2u}}{p^u}
&\ll \sum_p\frac{\abs{f(p)}^2}{p^{2-2\theta}}  \sum_{u=2}^\infty \frac{(u+1) (k-\tfrac{1}{r})^{2u}}{(u!)^2}\\
&\ll \sum_p\frac{\abs{f(p)}^2}{p^{2-2\theta}}<\infty,
\end{align*}
where the convergence of the last series follows from \eqref{eq:merten}
by partial summation, since $2-2\theta>1$. The contribution of the terms with 
$v = 1$ and $u \ge 1$ is
\begin{align*}
\sum_p\sum_{u=1}^\infty (u+2) \frac{\abs{f(p)}^2}{r^2 p} \frac{(k-\tfrac{1}{r})^{2u}
 \abs{f(p)}^{2u}}{p^u (u!)^2}
 &\ll\sum_p
  \frac{\abs{f(p)}^2}{p^{2-2\theta}} \sum_{u=1}^\infty \frac{(u+2) (k-\tfrac{1}{r})^{2u}}{(u!)^2}\\
&\ll\sum_p \frac{\abs{f(p)}^2}{p^{2-2\theta}}<\infty.
\end{align*}
The contribution of the terms with $v \ge 2, u \ge 1$  is
$$
\sum_p\sum_{v=2}^\infty \sum_{u=1}^\infty (v+u+1) 
\frac{\abs{f_{1/r}(p^v)}^2}{p^v} \frac{(k-\tfrac{1}{r})^{2u} \abs{f(p)}^{2u}}{p^u (u!)^2}.
$$
Using the inequality $v+u+1 \le (v+1)(u+1)$ for $v, u \ge 1$, we can write
\begin{align*}
\sum_p\sum_{v=2}^\infty \sum_{u=1}^\infty (v+u+1)
 &\frac{\abs{f_{1/r}(p^v)}^2}{p^v} \frac{(k-\tfrac{1}{r})^{2u} \abs{f(p)}^{2u}}{p^u (u!)^2} \\
&\le \sum_p \left( \sum_{v=2}^\infty (v+1) \frac{\abs{f_{1/r}(p^v)}^2}{p^v} \right)
 \left( \sum_{u=1}^\infty (u+1) \frac{(k-\tfrac{1}{r})^{2u}}{(u!)^2} \frac{\abs{f(p)}^{2u}}{p^u} \right) \\
&\ll \sum_p \sum_{v=2}^\infty (v+1) \frac{\abs{f_{1/r}(p^v)}^2}{p^v}<\infty,
\end{align*}
since $\tfrac{\abs{f(p)}^2}{p}\ll1$.
Hence we have 
$$
\sum_{l\le N} \frac{\abs{b(l)}^2}{l^{2\sigma}}\ll\exp\left(\sum_{p\le T}k^2\frac{\abs{f(p)}^2}{p^{2\sigma}}\right).
$$
For the lower bound, we restrict the summation over $l$ to the set $\mathcal{S}$ 
of square-free integers $l = mn$, where $m$ is composed entirely of primes 
$p \le T_1$, and $n \in \mathcal{N}$. Note that $l\le 4^{T_1}T^{1/9}\le N$
 for $l\in\mathcal{S}$. For such $l$, the divisors of $l$ lying in $\mathcal{N}$ 
 are exactly the divisors of $n$. Hence
$$
b(l)=\prod_{p \mid m}\frac{f(p)}{r}\prod_{p\mid n}\left(\frac{f(p)}{r}
+\left(k-\frac1r\right)f(p)\right)=\prod_{p \mid m}\frac{f(p)}{r}\prod_{p\mid n}kf(p),
$$
so that
$$
\abs{b(l)}^2 = \left(\prod_{p \mid m} \frac{1}{r^2}\abs{f(p)}^2\right) \left(\prod_{p \mid n} k^2 \abs{f(p)}^2\right).
$$
Hence, we get
\begin{align*}
\sum_{l\le N} \frac{\abs{b(l)}^2}{l^{2\sigma}}&\ge\sum_{l\in\mathcal{S}} 
\frac{\abs{b(l)}^2}{l^{2\sigma}}\ge\prod_{j=2}^\ell 
\left( \sum_{\substack{n \in \mathcal{N}_j \\ \mu^2(n)=1}} \frac{k^{2\Omega(n)} \abs{f(n)}^2}{n^{2\sigma}} \right).
\end{align*}
Dropping the condition $\Omega(n)\le K_j$ in the $j$-th factor introduces an error of at most
$$
e^{-K_j}\prod_{T_{j-1}<p\le T_j}\left(1+\frac{ek^2\abs{f(p)}^2}{p}\right)\le e^{-K_j+ek^2P_j}\le e^{-K_j/2},
$$
as $K_j\ge500(k+1)^2P_j$. Since $\sum_je^{-K_j/2}\ll1$, it follows that
\begin{align*}
\sum_{l\le N} \frac{\abs{b(l)}^2}{l^{2\sigma}}&\gg\prod_{T_1<p\le T_\ell}\left(1+k^2\frac{\abs{f(p)}^2}{p^{2\sigma}}\right)\gg\exp\left(k^2\sum_{p\le T_\ell}\frac{\abs{f(p)}^2}{p^{2\sigma}}\right).
\end{align*}
For $\sigma-\tfrac{1}{2}\ge \tfrac{c}{\log x }$ with any fixed  $c>0$, we have by partial summation that
\begin{align*}
\sum_{p\le x}\frac{\abs{f(p)}^2}{p^{2\sigma}}&=\frac{1}{x^{2\sigma-1}}
\sum_{p\le x}\frac{\abs{f(p)}^2}{p}+(2\sigma-1)\int_2^x\frac{1}{y^{2\sigma}}\left(\sum_{p\le y}\frac{\abs{f(p)}^2}{p} \right)dy\\
&=N_f\frac{\log\log x}{x^{2\sigma-1}}+N_f(2\sigma-1)\int_2^x\frac{\log\log y}{y^{2\sigma}}dy+O_c(1)\\
&=-N_f\log(2\sigma-1)+O_c(1).
\end{align*}
We apply this with $x=T$ for the upper bound and with $x=T_\ell$ for the
 lower bound, noting that $\log T_\ell\asymp\log T$. Together with \eqref{eq:merten}
in the case $\sigma=\tfrac{1}{2}$, this completes the proof of \Cref{lem:Llem}.
\end{proof}
We now deduce \Cref{lem:multimainlem}, from \Cref{lem:JKlem} 
and \Cref{lem:Llem}, following the argument of Heath-Brown \cite[p. 76]{HeathBrown1981}. 
\begin{proof}[Proof of \Cref{lem:multimainlem}]
If $\mathcal{K}(\tfrac{1}{2})\le T$, then by \eqref{eq:I} with 
$\sigma=\tfrac{1}{2}$ and \Cref{lem:Llem}, we get 
$$
\mathcal{J}(\tfrac{1}{2})\gg T(\log T)^{N_fk^2}.
$$
Henceforth we may assume that $\mathcal{K}(\tfrac{1}{2})\ge T$. 
Then $T^{\frac{2\sigma-1}{3}}\mathcal{K}(\tfrac{1}{2})^{\frac{4-2\sigma}{3}}\le
\mathcal{K}(\tfrac{1}{2})$, so that \Cref{lem:JKlem} with $N=T^{1/2}$ yields
\begin{equation}\label{eq:D}
\mathcal{K}(\sigma) \ll \mathcal{K}(\tfrac{1}{2}) T^{\frac{1-2\sigma}{6r}}.
\end{equation}
By \eqref{eq:I}, \Cref{lem:JKlem}, \eqref{eq:D} and \eqref{eq:II}, we get 
$$
\mathcal{L}(\sigma)\ll 
T^{\sigma-1/2}\mathcal{J}(\tfrac{1}{2})^{3/2-\sigma}+(\mathcal{L}(\tfrac{1}{2})
+\mathcal{J}(\tfrac{1}{2}))T^{\frac{1-2\sigma}{6r}}+e^{-T^2/(4r)}.
$$
Choosing $\sigma=\tfrac{1}{2}+\tfrac{c}{\log T}$ with $c\ge1$ and using \Cref{lem:Llem}, we get 
$$
T(\log T)^{N_fk^2}c^{-N_fk^2}\ll T(\log T)^{N_fk^2}e^{-\frac{c}{3r}}+e^{c}(\mathcal{J}(\tfrac{1}{2})+1),
$$
where the implied constant does not depend on $c$. Hence, for $c=c(f,k)>0$ sufficiently large, we see that 
$$
\mathcal{J}(\tfrac{1}{2})\gg T(\log T)^{N_fk^2}.
$$
It remains to deduce that $\mathcal{I}(T)\gg T(\log T)^{N_fk^2}$. However, $w(t)\ll 1$ for all $t$, and 
$$
w(t)\ll \exp(-(t^2+T^2)/(100r))
$$
for $t\notin[T,2T]$. Thus,
\begin{align*}
T(\log T)^{N_fk^2}\ll \mathcal{J}(\tfrac{1}{2})
&\ll \int_{T}^{2T}\abs{L(\tfrac{1}{2}+it,f)}^{2/r}\abs{\mathcal{N}(\tfrac{1}{2}+it,k-\tfrac{1}{r})}^2\,dt\\
&\qquad+~~e^{-T^2/(200r)}.
\end{align*}
Hence,
$$
\mathcal{I}(T)=\int_{T}^{2T}
\abs{L(\tfrac{1}{2}+it,f)}^{2/r}\abs{\mathcal{N}(\tfrac{1}{2}+it, k-\tfrac{1}{r})}^2\,dt\gg T(\log T)^{N_fk^2}.
$$
This completes the proof of \Cref{lem:multimainlem}.
\end{proof}
	
\section{LINEAR COMBINATIONS OF \texorpdfstring{$L$}{L}-FUNCTIONS}\label{sec:linearcombination}

We indicate how our method may be modified for linear combinations of $L$-functions by proving the following theorem.
	
\begin{thm}\label{thm:linearcombmultithm}
Let $f_1,\ldots, f_d$ be non-zero multiplicative functions and for $1 \le j \le d$, $L(s, f_j)$ be the 
$L$-function associated with the function $f_j$. Suppose that, for some $0<\theta<\tfrac{1}{2}$ and for 
all integers $r\ge 1$, each function $f_j$ satisfies the estimates 
$$
f_{j,1/r}(n) \ll_{f_j,r} n^\theta\quad\text{and}\quad\sum_p\sum_{m=2}^\infty\frac{\abs{f_{j,1/r}(p^m)}^2}{p^m}<\infty,
$$
where $f_{j,1/r}(n)$ is formally defined by
$$
L(s, f_j)^{1/r}=\sum_{n=1}^\infty\frac{f_{j,1/r}(n)}{n^s}.
$$
Assume, furthermore that, for some $N_j>0$, we have 
$$
\sum_{p\le x}\frac{\abs{f_j(p)}^2}{p}=N_j\log\log x+O(1)
\quad\text{as}\quad
x\to\infty.
$$
and that there exists a constant $B > 0$ such that for $i \neq j$, the following estimate holds for all $x\ge2$:
$$
\Re\left(\sum_{p\le x}\frac{f_i(p)\overline{f_j(p)}}{p}\right)\le B.
$$
Suppose also that each $L(s, f_j)$ admits a meromorphic continuation to the 
half-plane $\sigma\ge \tfrac{1}{2}$, with at most one pole at $s=1$, and that, 
for some $A>0$, we have $L(s, f_j)=O(\abs{t}^A)$ as $\abs{t}\to\infty$, 
uniformly for $\sigma\ge \tfrac{1}{2}$. Let $N=\max\{N_1,\ldots,N_d\}$. 
If $b_1,\ldots, b_d$ are non-zero complex numbers, then for all real $k>0$ 
and sufficiently large $T>0$, we have
$$
\int_T^{2T}\abs{\sum_{j=1}^d b_jL(\tfrac{1}{2}+it,f_j)}^{2k}\,dt\gg_{ (f_j), k, (b_j)} T(\log T)^{Nk^2}.
$$
\end{thm}
	
\begin{rmk}
\Cref{thm:linearcombmultithm} applies to the elements $L(s, f_1),\ldots, L(s, f_d)$ of $\mathcal{S}$ satisfying 
Selberg's orthonormality condition, that is,
$$
\sum_{p\le x}\frac{\abs{f_j(p)}^2}{p}=N_j\log\log x+O(1)
~~\text{ and }~~
\sum_{p\le x}\frac{f_i(p)\overline{f_j(p)}}{p} = O(1)
\quad\text{as}\quad x\to\infty
$$
for all $1 \le i, j \le 1$ with $i\neq j$.
\end{rmk}
	
We have the following corollary to \Cref{thm:linearcombmultithm}.
\begin{cor}\label{cor:linearcombcor}
Let $q>1$ be an integer and let $L(s,\chi)$ be the $L$-function associated with
a Dirichlet character $\chi$ modulo $q$.  If $(b_\chi)_{\chi\bmod{q}}$ are complex numbers, 
not all zero, then for all real $k>0$ and sufficiently large $T>0$, we have
$$
\int_T^{2T}\abs{\sum_{\chi\bmod{q}} b_\chi L(\tfrac{1}{2}+it,\chi)}^{2k}\,dt
~\gg_{q,k,(b_\chi)}~ T(\log T)^{k^2}.
$$
\end{cor}
Relabel $f_1,\ldots,f_d$ so that $N=N_1$. Let $C_{f_1,\ldots,f_d}>0$ be a sufficiently large constant 
depending only on $f_1,\ldots,f_d$ and $\ell$ denote the largest integer such that 
$\log_{\ell}T\ge C_{f_1,\ldots,f_d}$. 

Put $T_1=(k+1)^4C_{f_1,\ldots,f_d}$, and for $2\le j\le \ell$, define $T_j$, $\mathcal{P}_j(s)$, 
$P_j$, $K_j$, $\mathcal{N}_j$, $\mathcal{N}$, $g$, $\mathcal{N}_j(s,\alpha)$ 
and $\mathcal{N}(s,\alpha)$ as in \Cref{reqNot} with $f$ replaced with $f_1$.
	
Assume $k>0$ and let $r>1$ be the smallest integer such that $k>\tfrac{2}{r}$. Write 
$$
L(s)=\sum_{j=1}^d b_jL(s,f_j).
$$
Consider the quantity 
$$
\mathcal{I}(T)=\int_{T}^{2T}\abs{ L(\tfrac{1}{2}+it)}^{2/r}\abs{\mathcal{N}(\tfrac{1}{2}+it,k-\tfrac{1}{r})}^2\,dt.
$$
\Cref{thm:linearcombmultithm} then follows from the two lemmas below.
	
\begin{lem}\label{lem:linearcombmultithmNlem}
Let $T>0$ be sufficiently large. We have
$$
\int_T^{2T}\abs{\mathcal{N}\left(\tfrac{1}{2}+it,k-\tfrac{1}{r}\right)}^{\frac{2rk}{rk-1}}\,dt\ll T(\log T)^{N_1k^2}.
$$
\end{lem}
\begin{lem}\label{lem:linearcombmultithmmainlem}
Let $T>0$ be sufficiently large. We have
$$
\mathcal{I}(T)\gg T(\log T)^{N_1k^2}.
$$
\end{lem}
The proof of  \Cref{lem:multiNlem} is similar to that of \Cref{lem:linearcombmultithmNlem}. 
To prove \Cref{lem:linearcombmultithmmainlem}, we argue as follows. For $1\le j\le d$, 
define $a_j(l)$ by the formal identity
$$
L(s, f_j)^{1/r}\mathcal{N}(s,k-\tfrac{1}{r})=\sum_{l=1}^\infty\frac{a_j(l)}{l^s},
$$
so that
$$
a_j(l)=\sum_{\substack{mn=l\\ n\in\mathcal{N}}}f_{j,1/r}(m)(k-\tfrac{1}{r})^{\Omega(n)}g(n).
$$
Let 
$$
S_j(s)=\sum_{l\le T^{1/2}}\frac{a_j(l)}{l^s}
\phantom{mm}\text{and}\phantom{mm}
E(s)=L(s)\mathcal{N}(s,k-\tfrac{1}{r})^r-\sum_{j=1}^d b_jS_j(s)^r.
$$
Further, with $w(t)$ as in \Cref{sec:multimainlem}, let
$$
\mathcal{J}(\sigma)=\int_{-\infty}^{\infty}\abs{L(\sigma+it)}^{2/r}
\abs{\mathcal{N}(\sigma+it,k-\tfrac{1}{r})}^2w(t)\,dt,
$$
$$
\mathcal{K}(\sigma)=\int_{-\infty}^{\infty}\abs{E(\sigma+it)}^{2/r}w(t)\,dt,
$$
$$
\mathcal{L}(\sigma)=\int_{-\infty}^{\infty}\abs{\sum_{j=1}^d b_jS_j(\sigma+it)^r}^{2/r}w(t)\,dt.
$$
Following the proof of \Cref{lem:JKlem}, we get that $\mathcal{J}(\sigma)$ and $\mathcal{K}(\sigma)$ 
satisfy the following convexity estimates. Here $m$ is the maximum of the orders of the poles 
of the functions $L(s, f_j)$ at $s=1$.
\begin{lem}\label{lem:JK1lem}
Let $\tfrac{1}{2}\le \sigma\le \tfrac{3}{4}$ and $T>0$ be sufficiently large. Then 
$$
\mathcal{J}(\sigma)\ll T^{\sigma-1/2}\mathcal{J}(\tfrac{1}{2})^{3/2-\sigma}+e^{-T^2/(4r)},
$$
and
$$
\mathcal{K}(\sigma)\ll T^{\frac{2\sigma-1}{3}}T^{-\frac{2\sigma-1}{6r}}
\mathcal{K}(\tfrac{1}{2})^{\frac{4-2\sigma}{3}}+e^{-T^2/(4r)}.
$$
\end{lem}
Now it remains to estimate $\mathcal{L}(\sigma)$. Since 
$w(t)\gg1$ for $\tfrac{4T}{3}\le t\le\tfrac{5T}{3}$, applying the inequality 
$\abs{x+y}^{2/r}\ge \abs{x}^{2/r}-\abs{y}^{2/r}$, we get 
\begin{align*}
\mathcal{L}(\sigma)&\gg \int_{4T/3}^{5T/3}\abs{\sum_{j=1}^d b_jS_j(\sigma+it)^r}^{2/r}\,dt\\
&\ge {\abs{b_1}}^{2/r}\int_{4T/3}^{5T/3}\abs{ S_{1}(\sigma+it)}^{2}\,dt
-\sum_{j=2}^d{\abs{b_j}}^{2/r}\int_{4T/3}^{5T/3}\abs{ S_j(\sigma+it)}^{2}\,dt.
\end{align*}
Also, as in the proof of \Cref{lem:Llem}, we have
\begin{align*}
\mathcal{L}(\sigma)&\ll\sum_{j=1}^d\int_{0}^{3T}\abs{ S_j(\sigma+it)}^{2}\,dt.
\end{align*}
The proof of \Cref{lem:Llem} also yields the following lemma.
	
\begin{lem}\label{lem:diagonalmeanvalue}
Let $T>0$ be sufficiently large. We have
$$
\int_{T}^{2T}\abs{ S_{1}(\sigma+it)}^{2}\,dt\asymp T(\sigma-\tfrac{1}{2})^{-N_1k^2}
$$
uniformly for 
$
\frac{1}{2}+\frac{1}{\log T}\le \sigma\le 1.
$
Moreover,
$$
\int_{T}^{2T}\abs{ S_{1}(\tfrac{1}{2}+it)}^{2}\,dt\asymp T(\log T)^{N_1k^2}.
$$
\end{lem}
We now proceed to bound the ``off-diagonal" integrals.
\begin{lem}\label{lem:offdiagonalmeanvalue}
Let $T>0$ be sufficiently large. For $j\ge 2$, we have
$$
\int_{T}^{2T}\abs{ S_j(\sigma+it)}^{2}\,dt\ll T(\log T)^{N_j/r^2+N_1(k-1/r)^2},
$$
uniformly for $\tfrac{1}{2}\le \sigma\le 1.$
\end{lem}
\begin{proof}[Proof of \Cref{lem:offdiagonalmeanvalue}]
By the mean value theorem \cite[Cor. 3]{meanvalue}, we have 
$$
\int_{T}^{2T}\abs{ S_j(\sigma+it)}^{2}\,dt\ll T\sum_{l\le T^{1/2}}
\frac{\abs{a_j(l)}^2}{l}\le T\sum_{p\vert l\implies p\le T}\frac{\abs{a_j(l)}^2}{l}.
$$
For $2\le i\le \ell$, let $\mathcal{M}_i=\{l\in\mathbb{N}:p\mid l\implies T_{i-1}<p\le T_i\}$.
Further, for $l\in\mathcal{M}_i$, put
\begin{eqnarray*}
a_{j,i}(l)
&=&
\sum_{\substack{mn=l\\ n\in\mathcal{N}_i}}f_{j,1/r}(m)(k-\tfrac{1}{r})^{\Omega(n)}g(n) \\
\phantom{m}\text{and}\phantom{mmmm}
a^*_{j,i}(l)
&=&
\sum_{mn=l}f_{j,1/r}(m)(k-\tfrac{1}{r})^{\Omega(n)}g(n).
\end{eqnarray*}
We have 
\begin{align*}
\sum_{p\vert l\implies p\le T}\frac{\abs{a_j(l)}^2}{l}
&=\left(\sum_{p\vert l\implies p\le T_1}\frac{\abs{f_{j,1/r}(l)}^2}{l}\right)
\left(\sum_{p\vert l\implies T_\ell < p\le T}\frac{\abs{f_{j,1/r}(l)}^2}{l}\right)	\\
&\qquad\qquad\qquad\times\prod_{i=2}^\ell\left(\sum_{l\in\mathcal{S}_i}\frac{\abs{a_{j,i}(l)}^2}{l}\right).
\end{align*}
Arguing as in the proof of \Cref{lem:Llem} (cf. \cite[page 6]{HeapSoundararajan2022}), we have 
$$
\sum_{l\in\mathcal{S}_i}\frac{\abs{a_{j,i}(l)}^2}{l}=(1+O(e^{-K_i/2}))\sum_{l\in\mathcal{S}_i}\frac{\abs{a^*_{j,i}(l)}^2}{l}.
$$
Hence,
\begin{align*}
&\sum_{p\vert l\implies p\le T}\frac{\abs{a_j(l)}^2}{l}\\
&\ll\exp\left(\sum_{T_1<p\le T_\ell}\frac{\abs{f_j(p)/r+(k-\tfrac{1}{r})f_1(p)}^2}{p}+\frac{1}{r^2}
\sum_{T_\ell<p\le T}\frac{\abs{f_j(p)}^2}{p}\right)\\
&\ll\exp\bigg(\frac{1}{r^2}\sum_{p\le T}\frac{\abs{f_j(p)}^2}{p}+\Big(k-\frac{1}{r}\Big)^2
\sum_{T_1<p\le T_\ell}\frac{\abs{f_1(p)}^2}{p}\\
&\qquad\qquad\qquad\qquad\qquad\qquad+\frac{2}{r}\Big(k-\frac{1}{r}\Big)
\Re\sum_{T_1<p\le T_\ell}\frac{f_j(p)\overline{f_1(p)}}{p}\bigg).
\end{align*}
Since $\displaystyle{\Re\sum_{T_1<p\le x}\tfrac{f_j(p)\overline{f_1(p)}}{p}\le B+O(1)}$ for all $x>T_1$,
we have
$$
\sum_{p\vert l\implies p\le T}\frac{\abs{a_j(l)}^2}{l}\ll(\log T)^{N_j/r^2+N_1(k-1/r)^2}.
$$
This completes the proof of \Cref{lem:offdiagonalmeanvalue}.
\end{proof}
Since $N_j\le N_1$ and $k>\tfrac{1}{r}$, we have
$$
\frac{N_j}{r^2}+N_1\left(k-\frac1r\right)^2\le N_1k^2-\frac{2N_1}{r}\left(k-\frac1r\right)<N_1k^2,
$$
so that $T(\log T)^{N_j/r^2+N_1(k-1/r)^2}=o(T(\tfrac{\log T}{\log\log T})^{N_1k^2})$. Hence, from 
\Cref{lem:diagonalmeanvalue} and \Cref{lem:offdiagonalmeanvalue}, we deduce the following lemma.
\begin{lem}\label{lem:L1lem}	
Let $T>0$ be sufficiently large. We have
$$
\mathcal{L}(\sigma)\asymp T(\sigma-\tfrac{1}{2})^{-N_1k^2}
$$
uniformly for 
$
\frac{1}{2}+\frac{1}{\log T}\le \sigma\le \frac{1}{2}+\frac{\log\log T}{\log T}.
$
Moreover,
$$
\mathcal{L}(\tfrac{1}{2})\asymp T(\log T)^{N_1k^2}.
$$
\end{lem}
\Cref{lem:linearcombmultithmmainlem} then follows from \Cref{lem:JK1lem} and \Cref{lem:L1lem} in the same way
as \Cref{lem:multimainlem}.

\section{PROOFS OF COROLLARIES}\label{sec:corsec}
	
\begin{proof}[Proof of \Cref{cor:dedekindcor}] 
The Dedekind zeta function of a number field $K$ is defined by the Dirichlet series 
$$
\zeta_K(s) = \sum_{n=1}^\infty \frac{a_K(n)}{n^s},
$$
where $a_K(n)$ denotes the number of integral ideals of $K$ with norm $n$.
The function $\zeta_K(s)$ has the following properties (see \cite[Chap. 5]{iwaniec2004analytic}): 
\begin{enumerate}
\item 
$\zeta_K(s)$ admits a meromorphic continuation to $\mathbb{C}$, whose only
singularity is a simple pole at $s=1$. Moreover, there exists $A_K>0$
such that
$$
\zeta_K(s)\ll_{K} \abs{t}^{A_K}
$$
uniformly for $\sigma\ge \tfrac{1}{2}$ and $\abs{t}\ge 1$.
\item 
For $\sigma>1$, we have 
$$
\zeta_K(s)=\prod_{p}\prod_{j=1}^d\left(1-\frac{\alpha_{K,j}(p)}{p^s}\right)^{-1},
$$
where $d = [K : \mathbb{Q}]$ is the degree of $K$ over $\mathbb{Q}$
and $\alpha_{K,j}(p)$ are the local parameters of $\zeta_K(s)$ at $p$.
\item 
The local parameters satisfy $\abs{\alpha_{K,j}(p)}\le 1$.
\end{enumerate}
By the binomial theorem, for every integer $r \ge 1$, the coefficients 
$a_{K, 1/r}(n)$ of $\zeta_K(s)^{1/r}$ satisfy
$$
\abs{a_{K, 1/r}(n)}\le \tau_{d}(n)\ll_{K, \varepsilon} n^{\varepsilon}
$$
for all $\varepsilon>0$. It follows that 
$\sum_p\sum_{m=2}^\infty\abs{a_{K,k}(p^m)}^2 p^{-m} < \infty$.
By \cite[Lemma 4]{Krishnamoorthy2025}, we have
$$
\sum_{p\le x}\frac{a_K(p)^2}{p} 
~=~
M_K \log\log x + O(1).
$$
where $M_K=\#H\backslash G/H$, with $G=\operatorname{Gal}(L/\mathbb{Q})$ and
$H=\operatorname{Gal}(L/K)$ as in \Cref{cor:dedekindcor}. In particular, if
$K/\mathbb{Q}$ is a Galois extension, then $H$ is trivial and hence $M_K=\# G=d$.    
\end{proof}	
	
\begin{proof}[Proof of \Cref{cor:picor}] 
We write $L(s, \pi)=\displaystyle{\sum_{n=1}^{\infty}\lambda_{\pi}(n)n^{-s}}$. The $L$-functions 
under consideration have the following properties 
(see \cite[Chap. 5]{iwaniec2004analytic} and \cite{LRS}): 
\begin{enumerate}
\item 
$L(s,\pi)$ extends to an entire function of order $1$. Moreover, there exists $A_\pi > 0$
such that 
$$
L(s, \pi)\ll \abs{t}^{A_\pi}
$$
uniformly for $\sigma\ge \tfrac{1}{2}$ and $\abs{t}\ge 1$.
\item 
For $\sigma > 1$, we have 
$$
L(s, \pi)=\prod_{p}\prod_{j=1}^d\left(1-\frac{\alpha_{\pi , j}(p)}{p^s}\right)^{-1},
$$
where $d \ge 1$ is the degree of the Euler product and $\alpha_{\pi, j}(p)$ are the 
local parameters of $L(s, \pi)$ at $p$.
\item 
There exists $0<\gamma<\tfrac{1}{2}$ such that
$$
\abs{\alpha_{\pi,j}(p)}\le p^\gamma
$$
for all primes $p$ and all $1\le j\le d$.
\end{enumerate}

Let $r \ge 1$ be an integer. The coefficients $\lambda_{\pi, 1/r}(n)$ of $L(s,\pi)^{1/r}$ satisfy
$$
\abs{\lambda_{\pi, 1/r}(n)}\le n^\gamma \tau_{d}(n) \ll n^{\theta}
$$
for any fixed $\theta$ with $\gamma<\theta<\tfrac{1}{2}$. 
For $\mu\ge 1$, put 
$$
a_{\pi}(p^\mu)= \sum_{j}\alpha_{\pi , j}(p)^\mu,
$$ 
so that
$a_\pi(p)=\lambda_\pi(p)$. Consider the identity
$$
\sum_{m\ge0}\lambda_{\pi , k}(p^{m})x^{m}
=\prod_{j\le d}(1-\alpha_{\pi,j}(p)x)^{-k}
=\exp\Bigl(k\sum_{\mu\ge1}\frac{a_\pi(p^{\mu})}{\mu}x^{\mu}\Bigr).
$$
It follows that 
$$
\lambda_{\pi , k}(p^{m})=\sum_{\mathbf v}\prod_{\mu\ge1}\frac{1}{\mu^{v_\mu}v_\mu!}\prod_{\mu\ge1}
\bigl(k\,a_\pi(p^{\mu})\bigr)^{v_\mu},
$$   
where $\mathbf v$ runs over the sequences $(v_\mu)_{\mu\ge1}$ of non-negative integers, 
almost all zero, with $\sum_\mu\mu v_\mu=m$. Note that
$$
\sum_{\mathbf v}\prod_{\mu\ge1}\frac{1}{\mu^{v_\mu}v_\mu!}=1,
$$
this being the coefficient of $x^m$ in
$\exp\bigl(\sum_{\mu\ge 1}x^\mu/\mu\bigr)=(1-x)^{-1}$.
Hypothesis H of Rudnick--Sarnak (see \cite[page 15]{RS}) asserts that
$$
\sum_p\sum_{m=2}^{\infty}\frac{\abs{a_{\pi}(p^m)\log p}^2}{p^m} < \infty.
$$
Furthermore, Selberg orthogonality \cite[Corollary 1.5]{LWY} gives
\begin{equation}\label{eq:selberg}
\sum_{p\le x}\frac{\abs{\lambda_{\pi}(p)}^2}{p}=\log\log x+O(1)
\quad\text{as } x\to\infty.
\end{equation}
By Cauchy-Schwarz inequality, we have 
\begin{align*}
\sum_{m=2}^\infty \frac{\abs{\lambda_{\pi , 1/r}(p^m)}^2}{p^m}
&\le \sum_{m=2}^\infty \frac{1}{p^m} \sum_{\mathbf v}
\prod_{\mu\ge1}\frac{1}{\mu^{v_\mu}v_\mu!}\prod_{\mu\ge1}
\Bigl(\frac{\abs{a_\pi(p^{\mu})}^2}{r^2}\Bigr)^{v_\mu} \\
&= \exp\bigl(\frac{1}{r^2}\sum_{\mu\ge1}\frac{\abs{a_\pi(p^{\mu})}^2}{\mu p^\mu}\bigr)-1-\frac{\abs{a_\pi(p)}^2}{r^2p}\\
&= \exp\bigl(\frac{1}{r^2}\sum_{\mu\ge1}\frac{\abs{a_\pi(p^{\mu})}^2}{\mu p^\mu}\bigr)-1-\frac{1}{r^2}\sum_{\mu\ge1}\frac{\abs{a_\pi(p^{\mu})}^2}{\mu p^\mu}
+\frac{1}{r^2}\sum_{\mu\ge 2}\frac{\abs{a_\pi(p^{\mu})}^2}{\mu p^\mu}\\
& \ll \left(\sum_{\mu\ge 1}\frac{\abs{a_\pi(p^{\mu})}^2}{\mu p^\mu}\right)^2+\sum_{\mu\ge 2}
\frac{\abs{a_\pi(p^{\mu})}^2}{\mu p^\mu}\\
& \ll \frac{\abs{a_\pi(p)}^4}{p^{2}}+\sum_{\mu\ge 2}\frac{\abs{a_\pi(p^{\mu})}^2}{\mu p^\mu}.
\end{align*}
Summing over $p$ and using $\abs{a_\pi(p)}^2=\abs{\lambda_\pi(p)}^2\ll p^{2\theta}$,
we get
\begin{align*}
\sum_p\sum_{m=2}^{\infty} \frac{\abs{\lambda_{\pi , 1/r}(p^m)}^2}{p^m}
&\ll \sum_p \frac{\abs{\lambda_{\pi}(p)}^2}{p^{2-2\theta}}
+ \sum_p \sum_{\mu\ge 2}\frac{\abs{a_\pi(p^{\mu})}^2}{\mu p^\mu}
<\infty.
\end{align*}
Here the first sum converges by \eqref{eq:selberg} and partial summation, since
$2 - 2\theta >1$, and the second converges by Hypothesis H.
\end{proof}
	
\begin{proof}[Proof of \Cref{cor:linearcombcor}]
Without loss of generality, we may assume that all the $b_\chi$ are non-zero. 
For each Dirichlet character $\chi$ modulo $q$, the
function $L(s,\chi)$ extends to an entire function, except for a simple pole at
$s=1$ when $\chi$ is principal, and satisfies $L(s,\chi)\ll_q \abs{t}^{A}$ as
$\abs{t}\to\infty$ for some $A>0$, uniformly for $\sigma\ge\tfrac12$.

Since $\chi$ is completely multiplicative, for every integer $r\ge1$, the
coefficients of $L(s,\chi)^{1/r}$ are given by
$\chi_{1/r}(n)=\chi(n)\tau_{1/r}(n) \ll 1$ for all $n$. Hence
$\sum_p\sum_{m=2}^\infty\abs{\chi_{1/r}(p^m)}^2p^{-m}<\infty$.

For any Dirichlet character $\chi$ modulo $q$, we have 
$$
\sum_{p \le x} \frac{\abs{\chi(p)}^2}{p}
~=~ \sum_{\substack{p \le x \\ p \nmid q}} \frac{1}{p} 
~=~ \log\log x  ~+~ O(1),
$$
so that $N_\chi=1$ for all $\chi$.
If $\chi_1$ and $\chi_2$ are distinct characters modulo $q$, then
$\chi_1 \overline{\chi_2}$ is a non-principal Dirichlet character modulo $q$.
It follows from a classical theorem of Mertens that the series
$$
\sum_{p} \frac{\chi_1(p) \overline{\chi_2(p)}}{p}
$$
converges (see \cite[page 57]{Davenport}). Hence \Cref{cor:linearcombcor} follows from
\Cref{thm:linearcombmultithm}.
\end{proof}

\section{AI Declaration}
Authors acknowledges the use of Claude (Anthropic, Opus 5.5) for improvement of
language and literature survey.
	
\section{Acknowledgments}
The authors are grateful to the Institute of Mathematical Sciences, Chennai, for providing a 
conducive atmosphere during the project. The first and the third authors would also like to 
thank the Max Planck Institute for Mathematics, Bonn, where part of this project was carried out.

\end{document}